\documentclass{amsart}

\usepackage{amssymb}

\newtheorem{theorem}{Theorem}[section]
\newtheorem{lemma}[theorem]{Lemma}

\theoremstyle{definition}
\newtheorem{definition}[theorem]{Definition}
\newtheorem{example}[theorem]{Example}

\newtheorem{corollary}[theorem]{Corollary}
\newtheorem{proposition}[theorem]{Proposition}

\theoremstyle{remark}

\numberwithin{equation}{section}

\DeclareMathOperator{\PSL}{PSL}

\DeclareMathOperator{\SL}{SL}

\DeclareMathOperator{\CNS}{CNS}

\title{Separable boundaries for non-hyperbolic groups}

\begin{document}

\author{Jacopo Bassi}
\address{
Department of Mathematics, University of Tor Vergata, Via della Ricerca Scientifica 1, 00133 Roma, Italy}
\curraddr{}
\email{bssjcp01@uniroma2.it}
\thanks{}

\author{Florin R\u adulescu*}
\address{Department of Mathematics, University of Tor Vergata, Via della Ricerca Scientifica 1, 00133 Roma, Italy}
\curraddr{}
\email{radulesc@mat.uniroma2.it}
\thanks{\\
*Florin R\u adulescu is a member of the Institute of Mathematics of the Romanian Academy}


 
\begin{abstract} 
We exhibit examples of separable boundaries for non-hyperbolic groups. The main ingredient is the alignment property introduced by Furman in the study of rigidity properties of discrete subgroups of algebraic groups.
\end{abstract}

 
\maketitle
\bibliographystyle{amsplain}

\section*{Introduction}
Ozawa proved in his seminal paper \cite{ozawa} the existence of a canonical separable nuclear $C^*$-algebra containing the reduced $C^*$-algebra of the free group on $n\geq 2$ generators and contained in its injective envelope, giving an explicit, tight example of a deep theorem by Kirchberg. This algebra is the crossed product associated to the action of the group on its boundary. In his paper he conjectured the existence of such an intermediate nuclear $C^*$-algebra for every exact $C^*$-algebra. Kalantar and Kennedy proved in their outstanding work \cite{kk} that this conjecture holds true for the reduced $C^*$-algebra of any discrete exact group, by showing that the reduced crossed product associated to the action of the group on its Furstenberg boundary validates Ozawa's conjecture. The $C^*$-algebras produced in this way are not separable in the case of non-amenable groups and this fact is viewed there as a consequence of the vast generality to which the result applies, including non-hyperbolic groups. As already observed in \cite{kk} Remark 4.8, the proof of the "non-separable" Ozawa conjecture implies the "separable" Ozawa conjecture at the price that the $C^*$-algebra so obtained is no longer canonical. In the following we are interested in exhibiting a canonical nuclear \textit{separable} $C^*$-algebra $\CNS ( \Gamma)$ such that $C^*_r \Gamma \subset \CNS ( \Gamma) \subset I(C^*_r \Gamma)$ for a class of non-hyperbolic groups.\\
The proof by Ozawa relies on the hyperbolicity of the free group and the existence of a quasi-invariant doubly ergodic measure on its boundary, ingredients which allow a full description of the possible equivariant Borel maps from the boundary to the probability measures on it: there is only one. Moving to the non-hyperbolic setting, we shall consider an analogue of this boundary for certain non-hyperbolic groups, the key ingredient being the concept of alignment system introduced by Furman. Combining the approach by Ozawa and Furman's breakthroughs we show that canonical separable nuclear intermediate $C^*$-algebras exist for certain non-hyperbolic groups. Furman's techniques also allow to give examples of Kirchberg algebras arising in this way.\\
The examples considered here are probably known to experts. The authors think they provide interesting applications of the connection between Ozawa's techniques and Furman's work.

\section{Alignment systems and second countable boundaries} 

The alignment property was introduced by Furman in his inspiring work \cite{furman} as a fundamental tool in order to classify measurable centralizers, quotients and joinings of certain discrete subgroups of locally compact second countable groups and to obtain rigidity results for groups acting on spaces of horospheres, by means of dynamical methods. Following his techniques and the approach developed by Ozawa in \cite{ozawa}, we shall see in the present section that suitable alignment systems can be used to prove the existence of intermediate canonical nuclear separable $C^*$-algebras. We recall the main notion:

\begin{definition}[\cite{furman} Definition 3.1]
Let $\Gamma$ be a group, $(X,m)$ be a measure space with a measure class preserving $\Gamma$-action, $B$ a topological space with a continuous $\Gamma$-action and $\pi : X \rightarrow B$ be a measurable $\Gamma$-equivariant map. We shall say that $\pi$ has the \textit{alignment property} with respect to the $\Gamma$-actions if $x \mapsto \delta_{\pi(x)}$ is the only $\Gamma$-equivariant measurable map from $(X,m)$ to the space $\mathcal{P}(B)$ of all regular Borel probability measures on $B$. In such a situation $(\pi : (X,m) \rightarrow B; \Gamma)$ is an \textit{alignment system}.
\end{definition}

The first part of the following is shown in the proof of Proposition 3 of \cite{ozawa}.
\begin{lemma}
\label{lem1}
Let $\Gamma$ be a discrete countable group. Let $Y$ be a compact metric $\Gamma$-space and $X$ a locally compact $\Gamma$-space with a full-supported quasi-invariant Borel measure $m$. There is a bijection
\begin{equation*}
\{\Gamma -  \mbox{ ucp maps }: C(Y) \rightarrow L^\infty (X,m)\} \leftrightarrow \{ \Gamma -  \mbox{ meas. maps }: X \rightarrow \mathcal{P}(Y) \}.
\end{equation*}
\end{lemma}
\begin{proof}
Let $\phi : C(Y) \rightarrow L^\infty (X,m)$ be a $\Gamma$-equivariant ucp map. Let $\mathcal{C} \subset C(Y)$ be a $\Gamma$-invariant countably generated dense $*$-subalgebra of $C(Y)$. We want to lift $\phi$ to a $\Gamma$-equivariant linear positive continuous map $\tilde{\phi}$ from $\mathcal{C}$ with values in bounded measurable functions on $X$. For let $F=\{f_n\}_{n \in \mathbb{N}}$ be the countable set obtained by forming adjoints and products of generators for $\mathcal{C}$ and let $S_1 = \{ 1\}$; let $\tilde{\phi} (1) = 1$. Extend $\tilde{\phi}$ linearly to the linear span of $S_1$. Suppose now we are given a finite set $S_n \subset F$ and a continuous linear positive map $\tilde{\phi}$ from the linear span of $S_n$ to measurable functions on $X$ such that $\tilde{\phi}(1)=1$ modulo a null set. Let $f$ be an element of $F$ that is not in $S_n$. Let $\tilde{\phi}(f) = \sum_{i=1}^k  \alpha_i\tilde{\phi}(f_i)$ if $f=\sum_{i=1}^k \alpha_i f_i$ is a linear span of elements from $S_n$. Otherwise choose a measurable representative $\tilde{\phi}(f)$ such that $\|\tilde{\phi}(f)\|=\|\phi (f)\|$ and consider the linear extension of the map $\tilde{\phi}$ to the linear span of $S_{n+1} =S_n \cup \{f\}$ . For every finite set $\{f_{i_1},...,f_{i_k}\}$ of elements from $S_{n+1}$ let $N_{\{f_{i_1},...,f_{i_k}\}}=\{ x \in X \; | \; \exists (q_1,...,q_k) \in \mathbb{C}^k\; \mbox{ satisfying } \; \Re (q_i) \in \mathbb{Q}, \; \Im (q_i) \in \mathbb{Q} \; \forall 1 \leq i \leq k, \; |\sum_{j=1}^k q_j \tilde{\phi}(f_{i_j}) (x) | > \| \phi(\sum_{j=1}^k q_j f_{i_j})\|\}$. Being the countable union of null sets, $N_{\{f_{i_1},...,f_{i_k}\}}$ has measure zero. Suppose now that $x \in X$ is such that there are $\beta_1,..., \beta_k$ complex numbers satisfying  $|\sum_{j=1}^k \beta_j \tilde{\phi}(f_{i_j}) (x) | > \| \phi(\sum_{j=1}^k \beta_j f_{i_j})\|$; by continuity of the maps $(\beta_1,...,\beta_k) \mapsto |\sum_{j=1}^k \beta_j \tilde{\phi}(f_{i_j}) (x) |$ and $(\beta_1,...,\beta_k) \mapsto  \| \phi(\sum_{j=1}^k \beta_j f_{i_j})\|$ there is $(q_1,..., q_k) \in \mathbb{\mathbb{C}}^k$ with $\Re(q_i) \in \mathbb{Q}$, $\Im (q_i) \in \mathbb{Q}$ for every $1 \leq i \leq k$ such that $|\sum_{j=1}^k q_j \tilde{\phi}(f_{i_j}) (x) | > \| \phi(\sum_{j=1}^k q_j f_{i_j})\|$ and so $x$ belongs to $N_{\{f_{i_1},...,f_{i_k}\}}$. Hence $\{ x \in X \; | \; \exists \beta_1,..., \beta_k \in \mathbb{C} \; | \; |\sum_{j=1}^k \beta_j \tilde{\phi}(f_{i_j}) (x) | > \| \phi(\sum_{j=1}^k \beta_j f_{i_j})\} =N_{\{f_{i_1},...,f_{i_k}\}}$. Therefore, there is a set of full measure $E \subset X$ ($E= \cap_{\{f_{i_1},...,f_{i_k}\}\subset S_{n+1}} N_{\{f_{i_1},...,f_{i_k}\}}^c$) with the property that the linear map $\tilde{\phi} \cdot \chi_{E}$ is continuous on the linear span of $S_{n+1}$. By induction we find a continuous positive linear map $\tilde{\phi}$ from $\mathcal{C}$ taking values in measurable functions on $X$. Note now that for every $f \in F$ there is at most a null set $N' \subset X$ such that $\gamma^{-1} \tilde{\phi}(\gamma f) (x) \neq \tilde{\phi}(f)(x)$ for some $\gamma \in \Gamma$ and $x \in N'$. Again we can find a $\Gamma$-invariant set of full measure $E'$ such that the map $f \mapsto \tilde{\phi}(f) \cdot \chi_{E'}$ is continuous, linear, positive and $\Gamma$-equivariant. For every $x \in X$ extend the resulting functional $\Phi_x : f \mapsto \tilde{\phi} (f) (x)$ to the whole $C(Y)$ and let this extension be denoted by $\Phi_x$. We want to show that the resulting map $X \mapsto \mathcal{P}(Y)$, $x \mapsto \Phi_x$ is measurable.\\
Since $Y$ is a compact metric space, $\mathcal{P}(Y)$ is second countable and since $\mathcal{C}$ is dense in $C(Y)$, a basis of open sets is given by
\begin{equation*}
B_{\epsilon , \{f_i\}_{i=1}^n , \mu} := \{ \nu \in \mathcal{P}(Y) \; | \; | \nu (f_i) - \mu (f_i ) | < \epsilon \; \forall i=1,...,n\},
\end{equation*}
with $\epsilon >0$, $\{f_i\}_{i=1}^n \subset \mathcal{C}$, $\mu \in \mathcal{P}(Y)$. Since every open set is a countable union of sets of this form, it is enough to see that the inverse image under $\Phi$ of each such set is measurable in $X$. We have
\begin{equation*}
\begin{split}
\Phi^{-1} (B_{\epsilon, \{f_i\}_{i=1}^n , \mu} \cap \Phi (X))&= \{ x \in X \; | \; |\Phi_x (f_i) - \mu(f_i) | < \epsilon \; \forall i=1,...,n\}\\
&=\{ x \in X \; | \; |\tilde{(\phi (f_i))} (x) - \mu (f_i) | < \epsilon \; \forall i=1,...,n\}
\end{split}
\end{equation*}
for every choice of $\epsilon >0$, $\{f_i\}_{i=1}^n \subset \mathcal{C}$, $\mu \in \mathcal{P}(Y)$. But the function $x \mapsto |\tilde{(\phi (f_i))} (x) - \mu (f_i)|$ is measurable and one direction follows.

Let now $\Phi : X \rightarrow \mathcal{P}(Y)$ be a measurable $\Gamma$-equivariant map defined on the Borel set of full measure $E \subset X$; for $x \in E$ let
\begin{equation*}
\phi (f)(x):= \int f d \Phi_x.
\end{equation*}
The resulting map
\begin{equation*}
\phi : C(Y) \rightarrow L^\infty (X,m)
\end{equation*}
is unital, positive and contractive, hence it is ucp because of commutativity of $C(Y)$. \end{proof}

An action of a topological group $\Gamma$ on a compact space $X$ is \textit{strongly proximal} if the closure of the orbit of every element in $\mathcal{P}(X)$ contains a point mass. A minimal compact strongly proximal $\Gamma$-space is a \textit{$\Gamma$-boundary} (\cite{glasner_top} Definitions). The proof of the following is an adaptation of the proof of \cite{ozawa} Theorem 1. 

\begin{theorem}
\label{thm2}
Let $\Gamma$ be a discrete countable group. Let $Y$ be a compact metric $\Gamma$-space and $X$ a locally compact $\Gamma$-space with a full-supported quasi-invariant Borel measure $m$, $(\pi : (X,m) \rightarrow Y; \Gamma)$ an alignment system with $\pi$ a continuous surjection and suppose the action of $\Gamma$ on $X$ is amenable. There are natural $C^*$-algebra inclusions
\begin{equation*}
C^*_r \Gamma \subseteq C(Y) \rtimes_r \Gamma \subseteq I(C^*_r \Gamma).
\end{equation*}
In particular, $Y$ is a $\Gamma$-boundary and if the action of $\Gamma$ on $Y$ is amenable, then we can take $\CNS (\Gamma) = C(Y) \rtimes \Gamma$.
\end{theorem}
\begin{proof}
Both $C(Y) \rtimes_r \Gamma$ and $L^\infty (X,m) \rtimes_r \Gamma$ contain a copy of $C^*_r \Gamma$, say $C^*_r \Gamma_{(1)}$ and $C^*_r \Gamma_{(2)}$; let $\alpha: C^*_r \Gamma_{(1)} \rightarrow C^*_r \Gamma_{(2)}$ be the corresponding $*$-isomorphism. When restricted to $C_c (\Gamma , C(Y))$, this takes the form $\alpha (\sum_{\gamma \in \Gamma} \lambda^{(1)}_\gamma) = \sum_{\gamma \in \Gamma} \lambda^{(2)}_\gamma$, where $\lambda^{(1)}_\gamma$ is the corresponding unitary on $L^2 (\Gamma) \otimes H$ and $\lambda^{(2)}_\gamma$ the one on $L^2 (\Gamma) \otimes L^2 (X,m)$, where $H$ is a Hilbert space on which $C(Y)$ is represented faithfully. Let $\theta : C(Y)\rtimes_r \Gamma \rightarrow L^\infty (X,m) \rtimes \Gamma$ be a ucp map such that $\theta |_{C^*_r \Gamma_{(1)}} = \alpha$ and consider the composition $\phi := E \circ \theta |_{C(Y)} : C(Y) \rightarrow L^\infty (X,m)$, where $E : L^\infty (X,m) \rtimes \Gamma \rightarrow L^\infty (X,m)$ is the conditional expectation induced by the projection $p_e : L^2 (X,m) \otimes L^2 (\Gamma) \rightarrow L^2 (X,m) \otimes \mathbb{C} e$ ($e$ is the identity in $\Gamma$). Since $\theta |_{C^*_r \Gamma_{(1)}} = \alpha$, it follows that for every $y \in C(Y)\rtimes_r \Gamma$ and $x \in C^*_r \Gamma_{(1)}$ we have $\theta (xy)=\alpha (x) \theta (y)$ and $\theta(yx)=\theta (y) \alpha(x)$. Hence, for every $\gamma \in \Gamma$ and $f \in C(Y)$,
\begin{equation*}
\begin{split}
\phi (\gamma \cdot f) &= E(\theta (\lambda^{(1)}_\gamma f \lambda^{(1) \; *}_\gamma))\\
		&=E(\lambda^{(2)}_\gamma \theta (f) \lambda^{(2) \; *}_\gamma) = \gamma \cdot E(\theta (f)) = \gamma \cdot \phi (f),
\end{split}
\end{equation*}
where we used the fact that $E(\lambda^{(2)}_\gamma x \lambda^{(2) \; *}_\gamma)=\gamma \cdot E(x)$ for every $x \in L^\infty (X,m) \rtimes \Gamma$. Hence $\phi : C(Y) \rightarrow L^\infty (X,m)$ is a $\Gamma$-equivariant ucp map. By the alignment property and Lemma \ref{lem1} it follows that $\phi= \pi^*$. The faithfulness of $E$ implies that $\theta |_{C(Y)} = \pi^*$. Hence $\theta$ and $\pi^* \rtimes_r \Gamma$ both send elements in $C_c (\Gamma, C(Y))$ of the form $\sum_{\gamma \in \Gamma} f_\gamma \lambda^{(1)}_\gamma$ to $\sum_{\gamma \in \Gamma} \pi^* (f_\gamma) \lambda^{(2)}_\gamma$; then they coincide on $C(Y) \rtimes_r \Gamma$.\\
Since the action of $\Gamma$ on $X$ is amenable, $L^\infty (X,m) \rtimes \Gamma$ is injective and so there is a completely isometric copy of $I(C^*_r \Gamma)$, say $\phi (I(C^*_r \Gamma))$, contained in $L^\infty (X,m) \rtimes \Gamma$; moreover $\phi (I(C^*_r \Gamma))$ is the image of a projection $\eta : L^\infty (X,m) \rtimes \Gamma \rightarrow L^\infty (X,m) \rtimes \Gamma$, which is the identity on $C^*_r \Gamma_{(2)}$. Composing with the $*$-homomorphism $\pi^* \rtimes_r \Gamma : C(Y) \rtimes_r \Gamma \rightarrow L^\infty (X,m) \rtimes \Gamma$, we obtain a ucp map that, when restricted to $C^*_r \Gamma_{(1)}$, coincides with $\alpha$; by the previous argument, it coincides with $\pi^* \rtimes_r \Gamma$. In particular, the Choi-Effros product ($\star$) associated to $\eta$, is just the canonical product on $\pi^* \rtimes_r \Gamma (C(Y) \rtimes_r \Gamma)$ and since $\phi (I(C^*_r \Gamma))$ is the image of $\eta$, this gives the inclusions $\pi^* \rtimes_r \Gamma (C(Y) \rtimes_r \Gamma) \subset (\phi (C^*_r \Gamma), \star) \subset L^\infty (X,m) \rtimes \Gamma$. Note that $(\phi (C^*_r \Gamma), \star) \simeq I(C^*_r \Gamma)$ as $C^*$-algebras. The statement that $Y$ is a $\Gamma$-boundary follows from Theorem 3.4 of \cite{hamana}, arguing as in \cite{kk} Remark 5.6. \end{proof}

\begin{example}
The following example is based on the proof of \cite{furman} Theorem 4.8. Let $k=\mathbb{R}$ or $\mathbb{Q}_p$, with $p$ prime and $\Gamma < \SL_n (k)$ a lattice. Let $Q$ be the subgroup of upper triangular matrices, $Q= \{ (a_{i,j}) \; | \; i >j \Rightarrow a_{i,j} =0, \; 1 \leq i,j \leq n\}$ and $H$ the subgroup of unitriangular matrices, $H= \{ (a_{i,j}) \in Q \; | \; a_{i,i} =1, \; 1 \leq i \leq n\}$. Let also, for $1 \leq i<j \leq n$, $H_{i,j} = \{ 1 + t e_{i,j}, \; t \in k\}$, where $e_{i,j}$ is the matrix with $1$ on the $(i,j)$-entry and $0$ elsewhere. For every $1 \leq i<j \leq n$, the action of $H_{i,j}$ on $\SL_n (k)/Q$ is algebraic and since $H_{i,j} \simeq k$, it follows that $G/Q$ decomposes as the disjoint union of the fixed points $F_{i,j}$ for $H_{i,j}$ and the free orbits $V_{i,j}$. By Moore's Theorem (which follows from the Howe-Moore property proved in \cite{ciobotaru} and \cite{bekka} Theorem III.2.1) the action of $H_{i,j}$ on $G/\Gamma$ is ergodic and hence the same holds for the action of $\Gamma$ on $G/H_{i,j}$. The action of $H_{i,j}$ on $V_{i,j}$ is proper and by ergodicity the orbit of almost every point in every measurable set $E$ of positive measure intersects $E$, i.e. the action is conservative. Hence \cite{furman} Theorem 4.4 applies and $SL_n (k) /Q$ is a $\Gamma$-boundary. $Q$ is solvable, hence amenable and $\CNS (\Gamma) = C (\SL_n (k)/Q) \rtimes \Gamma$.
\end{example}

\begin{example}
\label{exmp2}
Let $S$ be a finite set of places and $\Gamma < \prod_{i\in S} \SL_2 (\overline{\mathbb{Q}}^i)$ a lattice with the property that for every $i\in S$, the image of $\Gamma$ under the projection on $\prod_{j\neq i} \SL_2 (\overline{\mathbb{Q}}^i)$ is dense in $\prod_{j\neq i} \SL_2 (\overline{\mathbb{Q}}^i)$. For every $i\in S$ denote $Q_i = \{ (a_{j,l}) \in \SL_2 (\overline{\mathbb{Q}}^i) \; | \; a_{2,1} =0\}$, $H_i = \{ (a_{j,l}) \in Q_i \; | \; a_{1,1}=a_{2,2}=1 \}$. Then every $H_i$ acts properly on $V_i = (\prod_{j\neq i} \SL_2 (\overline{\mathbb{Q}}^j)/Q_j) \times (\SL_2 (\overline{\mathbb{Q}}^i)/Q_i - \{ eQ_i\})$ and $G/Q - \{eQ\} = \bigcup_{i=0}^n V_i$. It follows from the denseness hypothesis and Moore Theorem that $\Gamma$ acts ergodically on $G/H_i$ for every $i$. It follows again from \cite{furman} Theorem 4.4 that $\prod_{i=0}^n \SL_2 (\overline{\mathbb{Q}}^i) /Q_i$ is a $\Gamma$-boundary. $\prod_{i=0}^n Q_i$ is amenable and so $\CNS (\Gamma) = C(\prod_{i=0}^n \SL_2 (\overline{\mathbb{Q}}^i)/Q_i) \rtimes \Gamma$.
\end{example}

\begin{corollary}
\label{cor1}
Let $S$ denote a finite set of places and let $\Gamma < \prod_{i \in S} \PSL_2 (\overline{\mathbb{Q}}^i)$ be a lattice with the property that, if $|S| >1$, then for every $i \in S$, the image of $\Gamma$ under the projection on $\prod_{j\neq i} \PSL_2 (\overline{\mathbb{Q}}^i)$ is dense in $\prod_{j\neq i} \PSL_2 (\overline{\mathbb{Q}}^i)$. Then $\Gamma$ is $C^*$-simple.
\end{corollary}
\begin{proof} 
Any non-trivial element in $\PSL(2,\mathbb{R})$ has at most two fixed points for its action on $\partial \mathbb{H}$ (and in this case it is hyperbolic). Let now $p$ be a prime number. Since the space $\PSL(2,\mathbb{Q}_p)/Q$ ($Q$ is the group of upper-triangular matrices) is isomorphic (as a $\PSL(2,\mathbb{Q}_p)$-space) to the projective space $P^1 (\mathbb{Q}_p)$, also the non-trivial elements of $\PSL(2,\mathbb{Q}_p)$ have at most two fixed points for their action on $\PSL(2,\mathbb{Q}_p)/Q$. The result follows from Example \ref{exmp2} and \cite{kk} Theorem 6.2. \end{proof}

\section{Kirchberg algebras associated to boundary actions} 

Let $\Gamma$ be a topological group and $X$ a $\Gamma$-space. The action of $\Gamma$ on $X$ is \textit{extremely proximal} if for every closed set $C \subsetneq X$ and every non-empty open set $U \subset X$ there is $\gamma \in \Gamma$ such that $\gamma C \subset U$ (\cite{glasner_top} Definitions). Since singleton sets are closed in Hausdorff spaces, it follows that extremely proximal actions are minimal. Infinite compact $\Gamma$-spaces associated to extremely proximal actions are boundaries (\cite{glasner_top} Theorem 2.3) and, in virtue of a fundamental result by Laca and Spielberg, under the assumption of topological freeness, the corresponding reduced crossed product $C^*$-algebras are purely infinite and simple (\cite{boundary} Theorem 5). Thanks to Jolissaint's deep insights, we now know that the same structure is shared by the crossed products associated to the more general class of topologically free actions satisfying the $n$-filling property (\cite{n-filling} Definition 0.1, Theorem 1.2 and Remark 1.3; see also the statement of Proposition \ref{prop3}). This fact allows us to show that some of the intermediate crossed products constructed in the previous section are Kirchberg algebras.

The following is based on the considerations given in the proof of \cite{furman} Theorem 4.4.
\begin{theorem}
\label{thm2}
Let $\Gamma$ be a discrete subgroup of a locally compact second countable group $G$. Suppose there are closed subgroups $H< Q < G$ with $Q$ cocompact such that $H$ leaves $G/Q - \{eQ\}$ invariant and acts properly on it. If the action of $\Gamma$ on $G/H$ is conservative then the action of $\Gamma$ on $G/Q$ is extremely proximal. In particular, if this action is topologically free, then $C(G/Q) \rtimes_r \Gamma$ is a simple purely infinite unital $C^*$-algebra; if it is also amenable, then $\CNS(\Gamma)=C(G/Q) \rtimes \Gamma$ is a Kirchberg algebra, hence it is classifiable by $K$-theoretic data.
\end{theorem}
\begin{proof}
Denote by $\pi : G/H \rightarrow G/Q$ the canonical quotient map and by $g: G/H \rightarrow G$ a Borel cross section of the quotient map $G \rightarrow G/H$. Let $K \subsetneq G/Q$ be a compact proper subset and let $C$ be a larger compact set satisfying $K \subset C^\circ \subset C \subsetneq G/Q$. It follows from Lusin Theorem and conservativity of the action of $\Gamma$ on $G/H$, arguing as in the proof of \cite{furman} Lemma 4.5, that there are $x \in \pi^{-1} (G/Q- C)$, $\gamma_n \rightarrow \infty$ in $\Gamma$ and $t_n \rightarrow \infty$ in $H$ satisfying
\begin{equation*}
\gamma_n x \rightarrow x, \qquad  g_{\gamma_n x} g_x^{-1} =: u_n \rightarrow e \in G, \qquad \gamma_n = u_n g_x t_n g_x^{-1}.
\end{equation*}
Since $g_x Q$ does not belong to $C$, $g_x^{-1} C$ is a compact subset of $G/Q -\{eQ\}$. Since $H$ acts properly on $G/Q - \{eQ\}$, there is $n_1 \in \mathbb{N}$ such that $t_n g_x^{-1} C \cap g_x^{-1} C = \emptyset$ for every $n > n_1$. Since $u_n \rightarrow e \in G$, there is $n_2 \in \mathbb{N}$ such that $u^{-1}_n K \subset C$ for every $n > n_2$. Hence, for $n > \max \{n_1,n_2\}$ we have $t_n g_x^{-1} K \cap g_x^{-1} u_n^{-1} K \subset t_n g_x^{-1} C \cap g_x^{-1} C =\emptyset$, which entails $\gamma_n K \cap K = \emptyset$. Let now $U \subset G/Q$ be open and $K \subsetneq G/Q$ be closed. If $K$ contains $U$ there are elements $\gamma_1$, $\gamma_2$ in $\Gamma$ such that $\gamma_1 K \subset G/Q - K$ and $\gamma_2 (G/Q - U) \subset U$ and so $\gamma_2 \gamma_1 K \subset U$; otherwise there is $\gamma \in \Gamma$ such that $\gamma (G/Q - (U - K)) \subset U$. It follows that the action of $\Gamma$ on $G/Q$ is extremely proximal. Hence $C(G/Q) \rtimes_r \Gamma$ is purely infinite and simple by \cite{boundary} Theorem 5. If the action is amenable, then $C(G/Q) \rtimes \Gamma$ is a Kirchberg algebra and it is classifiable in virtue of \cite{tu} Lemma 3.5, Theorem 4.2.4 and \cite{phillips_k} Theorem 4.2.4. \end{proof}

\begin{example}[cfr. \cite{furman} Corollary 4.6]
\label{exmp3}
Let $k$ be either $\mathbb{R}$ or $\mathbb{Q}_p$ for $p$ prime and $\Gamma < \PSL_2 (k)$ a lattice. As usual, we let $Q = \{ (a_{i,j}) \in \SL_2 (k) \; | \; a_{2,1} =0\}$ and $H= \{ (a_{i,j}) \in Q \; | \; a_{1,1}=a_{2,2} =1\}$. Then $H \simeq k$ acts by translation on $G/Q - \{eQ\} \simeq k$, hence properly. $Q$ is amenable and $\CNS (\Gamma) = C(G/Q) \rtimes \Gamma$ is a Kirchberg algebra.
\end{example}

\begin{proposition}
\label{prop3}
Let $\Gamma$ be a dense subgroup of a locally compact second countable group $G$, $X$ be an infinite compact Hausdorff space on which $G$ acts and $n \in \mathbb{N}$. Suppose the action of $G$ on $X$ is $n$-filling. Then for every $n \in \mathbb{N}$ and non-empty open sets $U_1,..., U_n$ in $X$, there are elements $\gamma_1,..., \gamma_n \in \Gamma$ such that $\bigcup_{i=1}^n \gamma_i U_i = X$, i.e. the action of $\Gamma$ on $X$ is $n$-filling.
\end{proposition}
\begin{proof}
Let $U_1,..., U_n$ be non-empty open sets and $g_1,...,g_n \in G$ such that $\bigcup_{i=1}^n g_i U_i = X$. There is a neighborhood of the identity $V \subset G$ such that $g_1 V U_1^c \cap (\bigcap_{i=2}^n g_i U_i^c) = \emptyset$. Hence there is $\gamma_1 \in \Gamma$ such that $\gamma_1 U_1 \cup \bigcup_{i=2}^n g_i U_i = X$. Let $k \in \{1,...,n-1\}$ and $\gamma_1,..., \gamma_k \in \Gamma$ be given such that $(\bigcup_{i=1}^k \gamma_i U_i) \cup (\bigcup_{i >k} g_i U_i) =X$. Let $U_{n+1} := \emptyset$. Let $W \subset G$ be a neighborhood of the identity such that $(\bigcap_{i=1}^k \gamma_i U_i^c) \cap (g_{k+1} W U_{k+1}^c) \cap (\bigcap_{i>k+1} g_i U_i^c) = \emptyset$. Then there is $\gamma_{k+1} \in \Gamma$ such that $(\bigcap_{i=1}^{k+1} \gamma_i U_i^c) \cap (\bigcap_{i>k+1} g_i U_i^c)=\emptyset$. Then $\bigcup_{i=1}^n \gamma_i U_i = X$. \end{proof}

Since $2$-filling actions are exactly the extremely proximal ones, it follows from Proposition \ref{prop3} that the lattices $\Gamma < \prod_{i\in S} \SL_2 (\overline{\mathbb{Q}}^i)$ considered in Example \ref{exmp2} admit extremely proximal topologically free actions on the compact spaces $\SL_2 (\overline{\mathbb{Q}}^i) / Q_i$ for every $i\in S$ (the fact that these actions are topologically free was observed in the proof of Corollary \ref{cor1}), but the associated crossed product algebras will not be nuclear. Indeed suppose there is $i \in S$ such that the crossed product $C(\SL_2 (\overline{\mathbb{Q}}^i)/Q_i) \rtimes \Gamma$ is nuclear, then the action of $\Gamma$ on $\SL_2 (\overline{\mathbb{Q}}^i)/Q_i$ would be amenable (\cite{delaroche2} Theorem 4.5). In particular, for every quasi-invariant finite measure $\mu$ on $\SL_2 (\overline{\mathbb{Q}}^i)/Q_i$ there would be a $\Gamma$-equivariant norm $1$ projection $p : l^\infty (\Gamma, L^\infty (\SL_2 (\overline{\mathbb{Q}}^i)/Q_i, \mu)) \rightarrow L^\infty (\SL_2 (\overline{\mathbb{Q}}^i)/Q_i, \mu)$ (\cite{delaroche2} Lemma 4.3) and it would follow from the ergodicity of the action of $\Gamma$ on $\SL_2 (\overline{\mathbb{Q}}^i)/Q_i$ and \cite{kuhn} that the Koopman representation associated to such a quasi-invariant measure is weakly contained in the left regular representation; but then there should be a non-zero vector $\xi$ in $L^2(\SL_2 (\overline{\mathbb{Q}}^i)/Q_i, \mu)$ whose associated matrix coefficient is in $L^{2+\epsilon}(\Gamma)$ for every $\epsilon >0$ (\cite{cowling-haagerup} Theorem 2). This is impossible since there is a sequence $\gamma_n$ in $\Gamma$ converging to the identity in $\SL_2 (\overline{\mathbb{Q}}^i)$. On the other hand, if $\Gamma$ is a lattice in $\prod_{i\in S} \PSL_2 (\overline{\mathbb{Q}}^i)$ and its action on $\prod_{i=0}^n \SL_2 (\overline{\mathbb{Q}}^i)/  Q_i$ is locally contractive (by which we mean that there is a group element with an attracting fixed point in the sense of \cite{n-filling} Definition 2.3) , it follows again from Example \ref{exmp2} (which entails minimality) and \cite{n-filling} Proposition 2.5 that $\CNS (\Gamma) = C(\prod_{i\in S} \SL_2 (\overline{\mathbb{Q}}^i) /  Q_i) \rtimes \Gamma$ is a Kirchberg algebra.

\begin{proposition}
\label{prop4}
Let $S$ denote a finite set of places and let $\Gamma < \prod_{i \in S} \SL_2 (\overline{\mathbb{Q}}^i)$ be a subgroup containing a diagonal copy of the group of diagonal matrices in $\SL_2$ with entries in $\mathbb{Z}[1/\prod_{i \in S\backslash \infty} p_i]$ ($\infty$ is the Archimedean place: $\overline{\mathbb{Q}}^\infty = \mathbb{R}$). Then the action of $\Gamma$ on $\prod_{i \in S} \SL_2 (\overline{\mathbb{Q}}^i) / Q_i$ is locally contractive.
\end{proposition}
\begin{proof}
For every $i \in S$, identify $\SL_2 (\overline{\mathbb{Q}}^i) / Q_i$ with the projective line $P^1 (\overline{\mathbb{Q}}^i)$ and for $i \in S \backslash \infty$ let $p_i$ be the corresponding prime number. The sequence $a_k :=(\prod_{i \in S\backslash \infty} p_i)^k$ converges to zero for every non-Archimedean place in $S$ and diverges in $\mathbb{R}$.  Hence, if $\infty \in S$, the point $[1:0] \times \prod_{i \in S\backslash \infty} [0:1] \in \SL_2 (\mathbb{R})/Q_\infty \times \prod_{i \in S \backslash \infty} \SL_2 (\overline{\mathbb{Q}}^i)/Q_i$ admits a neighborhood that is contracted under the action of the diagonal matrix with entries $a_{1,1} = a_1$, $a_{2,2} = a_1^{-1}$. If $\infty \notin S$, the same conclusion holds for the point $\prod_{i \in S} [0:1]$. \end{proof}

\begin{example}
Let $p$ be a prime number. By Example \ref{exmp2} the space $\partial \mathbb{H} \times P^1 (\mathbb{Q}_p)$ is a boundary for $\SL_2 (\mathbb{Z}[1/p])$. It follows from Proposition \ref{prop4} that \\$C(\partial \mathbb{H} \times P^1 (\mathbb{Q}_p)) \rtimes \PSL_2 (\mathbb{Z}[1/p])$ is a Kirchberg algebra (in the UCT class).
\end{example}

\section{Acknowledgment}
The authors thank the anonymous referee for his/her precious comments on a previous version of the manuscript. They acknowledge the support of INdAM-GNAMPA and the MIUR Excellence Department Project awarded to the Department of Mathematics, University of Rome Tor Vergata, CUP E83C180001000 and of the grant Beyond Borders: "A geometric approach to harmonic analysis and spectral theory on trees and graphs", CUP: E89C20000690005. The first named author was supported by the ERC grant n. 669240 QUEST "Quantum Algebraic Structures and Models", CUP: E52I15000700002, by the MIUR - Excellence Departments - grant: "$C^*$-algebras associated to $p$-adic groups, bi-exactness and topological dynamics", CUP: E83C18000100006 and by the grant Beyond Borders: "Interaction of Operator Algebras with Quantum Physics and Noncommutative
Structure", CUP: E84I19002200005, during the period of this research. The second named author acknowledges the partial  support of the grant  "The convex space of sofic representations", CNCS Romania, PN-III-P1-1.1-TE-2019-0262. The present project is part of: - OAAMP - Algebre di operatori e applicazioni a strutture non commutative in matematica e fisica, CUP E81I18000070005. Florin R\u adulescu is a member of the Institute of Mathematics of the Romanian Academy.

\end{document}